\documentclass[a4paper]{amsart}

\usepackage{wasowicz_hyperref}
\usepackage{times}

\renewcommand{\ge}{\geqslant}

\info{Manuscript}

\begin{document}

\date{\today}

\title[On the classes of higher-order Jensen- and Wright- convex functions II]{On the classes of higher-order Jensen-convex functions and Wright-convex functions, II}

\address{Department of Mathematics, University of Bielsko-Bia{\l}a, Willowa 2, 43--309 Bielsko--Biała, Poland}

\author{Jacek Mrowiec}
\author{Teresa Rajba}
\author{Szymon Wąsowicz}
\email[Jacek Mrowiec]{jmrowiec@ath.bielsko.pl}
\email[Teresa Rajba]{trajba@ath.bielsko.pl}
\email[Szymon Wąsowicz]{swasowicz@ath.bielsko.pl}

\begin{abstract}
Recently Nikodem, Rajba and Wąsowicz compared the classes of $n$-Wright-convex functions and $n$-Jensen-convex functions by showing that the first one is a~proper subclass of the latter one, whenever $n$ is an odd natural number. Till now the case of even~$n$ was an open problem. In this paper the complete solution is given: it is shown that the inclusion is proper for any natural~$n$. The classes of strongly $n$-Wright-convex and strongly $n$-Jensen-convex functions are also compared (with the same assertion).
\end{abstract}


\maketitle

\section{Introduction}

Let $\I\subset\R$ be an interval and $f\colon\I\to\R$. The \emph{difference operator} is given by
\[
 \Delta_hf(x)=f(x+h)-f(x)
\]
for any $x\in\I$ and $h\in\R$ \st\ $x+h\in\I$. In 1954 Wright~\cite{Wri54} considered the class of functions~$f:\I\to\R$, for which the function $\Delta_h f(\cdot)$ is non-decreasing for any $h>0$ small enough to guarantee that all the involved arguments belong to~$\I$. This monotonicity property is equivalent to the condition
\begin{equation}\label{eq:Wright}
 f\bigl(tx+(1-t)y\bigr)+f\bigl((1-t)x+ty\bigr)\xle f(x)+f(y)\quad \big(x,y\in\I,\;t\in[0,1]\big).
\end{equation}
Nowadays a~function fulfilling~\eqref{eq:Wright} is called to be \emph{Wright-convex}. For $t=\frac{1}{2}$ we get immediately
\[
 f\Bigl(\frac{x+y}{2}\Bigr)\xle\frac{f(x)+f(y)}{2}\quad\bigl(x,y\in\I\bigr),
\]
so any Wright-convex function is necessarily Jensen-convex. It was shown by Ng~\cite{Ng87} that any Wright-convex function is a~sum of the convex function (in the usual sense) and the additive one. Such representation leads to the simple (and now classical) proof that the inclusion in question is proper. Namely, if $a\colon\R\to\R$ is a~discontinuous additive function and $f(x)=|a(x)|$, then $f$~is discontinuous and Jensen-convex. Hence the graph of~$f$ is not dense on the whole plane. But it is well-known (\cf~\eg~\cite{Kuc09}) that the graphs of discontinuous additive functions (and also -- by Ng's representation -- the graphs of discontinuous Wright-convex functions) mapping~$\R$ into~$\R$ are dense on the whole plane. That is why~$f$ is not Wright-convex.
\par\medskip
Recently Nikodem, Rajba and Wąsowicz in the paper~\cite{NikRajWas12JMAA} considered the related comparison problem for the classes of $n$-Wright-convex functions and $n$-Jensen-convex functions (the explanation is given below in this section). If $n\in\N$ is odd, they have shown that the first class is the proper subclass of the latter one. The example of $n$-Jensen-convex function which is not $n$-Wright-convex was $f\colon\R\to\R$ \st\ $f(x)=\bigl(\max\{a(x),0\}\bigr)^n$, where $a\colon\R\to\R$ is the certain discontinuous additive function. Next P\'ales~\cite{Pal13} proved that this example works for any discontinuous additive function. The authors of~\cite{NikRajWas12JMAA} did not solve the  problem, if~$n$ is even. The objective of the present paper is to give the complete solution by showing that the inclusion is proper for all $n\in\N$. Moreover, the construction of our example does not depend on the parity of~$n$. Then for odd~$n$ we have a~solution different from the ones given in the papers~\cite{NikRajWas12JMAA,Pal13}.
\par\medskip
Let us recall the concepts of higher-order Jensen-convex (coming from Popoviciu~\cite{Pop34}) and Wright-convex functions (introduced by Gil\'anyi and P\'ales~\cite{GilPal01}).
\par\medskip
The iterates of ~difference operator are we defined as usual: for $n\in\N$
\[
 \Delta_{h_1\dots h_nh_{n+1}}f(x)=\Delta_{h_1\dots h_n}\bigl(\Delta_{h_{n+1}}f(x)\bigr)
\]
provided that all the involved arguments belong to~$\I$. If $h_1=\dots=h_n=h$, then we write 
\[
 \Delta_h^nf(x)=\Delta_{h\dots h}f(x).
\]
In particular, $\Delta_h^1f(x)=\Delta_hf(x)$. 
Let $n\in\N$ and $f\colon\I\to\R$ be a~function. We say that $f$ is \emph{Wright-convex of order~$n$} (or $n$-\emph{Wright-convex} for short), if
\[
 \Delta_{h_1\dots h_{n+1}}f(x)\xge 0
\]
for any $x\in\I$ and $h_1,\ldots,h_{n+1}>0$ \st\ $x+h_1+\dots+h_{n+1}\in\I$. If the above inequality is required only for $h_1=\dots=h_{n+1}=h>0$, precisely, if
\[
 \Delta_h^{n+1}f(x)\xge 0
\]
for all $x\in I$ and $h>0$ such that $x+(n+1)h\in\I$, then~$f$ is called \emph{Jensen-convex of order~$n$} ($n$-\emph{Jensen-convex} for short). It is easy to see that $1$-Wright-convexity and $1$-Jensen-convexity are equivalent to Wright-convexity and Jensen-convexity, respectively.
\par\medskip
By the above definitions any $n$-Wright-convex function is necessarily $n$-Jensen-convex. As we announced above, in this paper we show that this inclusion is proper.

\section{Main result}
Let us start with two preparatory notes. The following equations hold for a~power function function $f(x)=x^n$ (\cf~\cite[Lemma 15.9.2]{Kuc09}):
\begin{equation}\label{eq:delta-xn}
 \Delta_{h_1\ldots h_n}\bigl(x^n\bigr)=n!\cdot\prod_{i=1}^nh_i \qquad\text{and}\qquad \Delta_h^n\bigl(x^n\big)=n!\cdot h^n.
\end{equation}
\par\medskip
The formula below was given in~\cite{NikRajWas12JMAA} (as a~part of the proof of Corollary~2.2).
\begin{lem}\label{lemat2}
Let $a\colon\R\to\R$ be an additive function, $g\colon\R\to\R$ and $n\in\N$.
Then
\[
 \Delta_{h_1\dots h_n}(g\circ a) (x)=\Delta_{a(h_1)\dots\,a(h_n)}g\bigl(a(x)\bigr)
\]
for any $x,h_1,\dots,h_n\in\R$.
\end{lem}
Now we are in a position to state our main result.
\begin{thm}\label{th_main}
For any $n\in\N$ the class of $n$-Wright-convex functions is a~proper subclass of the class of $n$-Jensen-convex functions.
\end{thm}

\begin{proof}
 Because every $n$-Wright-convex function is $n$-Jensen-convex, it is enough to construct an~$n$-Jensen-convex function $f\colon\R\to\R$ which is not $n$-Wright convex. To this end consider a~Hamel basis $H$ of the linear space $\R$ over~$\Q$ \st~$h_0=1\in H$. Any $x\in\R$ is uniquely represented as a linear combination
\[
 x=\lambda_0h_0+\sum_{t}\lambda_th_t,
\]
where $\lambda_0,\lambda_t\in\Q$, $h_t\in H$. Obviously the functions $\alpha(x)=\lambda_0h_0$ and $\beta(x)=x-\alpha(x)$ are additive. Define the function $f\colon\R\to\R$ by
\[
 f(x)=\bigl(\alpha(x)\bigr)^{n+1}+\bigl(\beta(x)\bigr)^{n+1}.
\]
Since the difference operator is additive, we infer by Lemma~\ref{lemat2} and the equation~\eqref{eq:delta-xn} that
\begin{multline}\label{eq:dj}
\Delta_{h_1\dots h_{n+1}}f(x)=\Delta_{\alpha(h_1)\ldots\alpha(h_{n+1})}\bigl(\alpha(x)\bigr)^{n+1}+\Delta_{\beta(h_1)\ldots\beta(h_{n+1})}\bigl(\beta(x)\bigr)^{n+1}\\
=(n+1)!\biggl(\prod_{i=1}^{n+1}\alpha(h_i)+\prod_{i=1}^{n+1}\beta(h_i)\biggr).
\end{multline}
In particular,
\[
 \Delta_h^{n+1}f(x)=(n+1)!\Bigl(\bigl(\alpha(h)\bigr)^{n+1}+\bigl(\beta(h)\bigr)^{n+1}\Bigr).
\]
Let $h>0$. Because of the representation $h=\alpha(h)+\beta(h)$, it is easy to see that $\bigl(\alpha(h)\bigr)^{n+1}+\bigl(\beta(h)\bigr)^{n+1}>0$. Therefore $\Delta_h^{n+1}f(x)>0$ which shows that~$f$ is $n$-Jensen-convex.

To prove that $f$ is not $n$-Wright-convex take now $h_1=-1+\sqrt{2}$, $h_2=1$ and (if $n\ge 2$) $h_3=\dots=h_{n+1}=1+\sqrt{2}$. Then $h_i>0$, $i=1,\dots,n+1$. Moreover $\alpha(h_1)=-1$, $\alpha(h_i)=1$,  $i=2,\dots,n+1$ and $\beta(h_2)=0$. By~\eqref{eq:dj} we obtain 
\[
 \Delta_{h_1\ldots h_{n+1}}f(x)=-(n+1)!<0.
\]
It means that~$f$ is not $n$-Wright-convex.
\end{proof}

\section{The classes of higher order strongly Jensen-convex functions and Wright-convex functions}

Let $n\in\N$ and $c>0$. A~function $f\colon\I\to\R$ is called \emph{strongly Wright-convex of order $n$ with modulus $c$} (or strongly $n$-Wright-convex with modulus $c$) if
\[
\Delta_{h_1\ldots h_{n+1}}f(x)\xge c (n+1)!\,h_1\ldots h_{n+1}
\]
for all $x\in I$ and $h_1,\ldots,h_{n+1}>0$ such that $x+h_1+\dots+h_{n+1}\in\I$. If the above inequality is required only for $h_1=\dots=h_{n+1}=h>0$, precisely, if
\[
 \Delta_{h}^{n+1}f(x)\xge c(n+1)!\,h^{n+1}
\]
for all $x\in I$ and $h>0$ such that $x+(n+1)h\in\I$, then $f$ is called \emph{strongly Jensen-convex of order $n$ with modulus~$c$} (or strongly $n$-Jensen-convex with modulus $c$) (\cf~ \cite{GerNik11}).
For $c=0$ we arrive at standard $n$-Wright-convex and $n$-Jensen-convex functions, respectively.
\par\medskip
The following characterization of strongly $n$-Wright-convex functions was recently given in~\cite{GilMerNikPal15}.
\begin{thm}\label{tw4}
Let $n\in\N$ and $c>0$. A function $f\colon\I\to\R$ is strongly $n$-Wright-convex with modulus~$c$ if and only if the function $g\colon\I\to\R$ \st\ $g(x)=f(x)-cx^{n+1}$, is $n$-Wright-convex.
\end{thm}
It is surprising that so far nobody wrote explicitly the similar characterization of strongly $n$-Jensen-convex functions. We fill this gap now.
\begin{thm}\label{tw5}
Let $n\in\N$ and $c>0$. A function $f\colon\I\to\R$ is strongly $n$-Jensen-convex with modulus~$c$ if and only if the function $g\colon\I\to\R$ \st\ $g(x)=f(x)-cx^{n+1}$, is $n$-Jensen-convex.
\end{thm}
\begin{proof}
Suppose first, that the function $f$ is strongly $n$-Jensen-convex with modulus $c$. Applying~\eqref{eq:delta-xn} to $g$ we infer that
\[
\Delta_{h}^{n+1}g(x)= \Delta_{h}^{n+1}f(x) - \Delta_{h}^{n+1}(c x^{n+1}) \xge c (n+1)!\,h^{n+1}- c (n+1)!\,h^{n+1}=0,
\]
whence $g$ is $n$-Jensen-convex.

Conversely, if $g$ is an $n$-Jensen-convex with modulus~$c$, use~\eqref{eq:delta-xn} for $f(x)=g(x)+cx^{n+1}$:
\[
\Delta_{h}^{n+1}f(x)= \Delta_{h}^{n+1}g(x) + \Delta_{h}^{n+1}(c x^{n+1}) \xge 0+ c (n+1)!\,h^{n+1}=c (n+1)!\,h^{n+1}.
\]
This proves that~$f$ is strongly $n$-Jensen-convex.
\end{proof}
The following result we derive from Theorems \ref{th_main}, \ref{tw4}, \ref{tw5}.
\begin{thm}
Let $n\in\N$ and $c>0$. The class of strongly $n$-Wright-convex functions with modulus~$c$ is a~proper subclass of the class of strongly $n$-Jensen-convex functions with modulus~$c$.
\end{thm}
\begin{proof}
It is enough to show, that there exists a strongly $n$-Jensen--convex function with modulus $c$, which is not strongly $n$-Wright--convex with modulus $c$. By Theorem~\ref{th_main} there exists a function~$g$, which is $n$-Jensen-convex and $g$ is not $n$-Wright-convex. Let $f(x)=g(x)+cx^{n+1}$. Since $g$ is $n$-Jensen-convex, it follows by Theorem~\ref{tw5} that $f$~is strongly $n$-Jensen-convex with modulus~$c$. By Theorem~\ref{tw4}~$f$ is not strongly $n$-Wright-convex with modulus~$c$ (because the function $g(x)=f(x)-cx^{n+1}$ is not $n$-Wright-convex). This concludes the proof.
\end{proof}

\end{document}